\documentclass[a4paper,12pt]{article}
\usepackage[utf8]{inputenc} 
\usepackage[english]{babel}
\usepackage{amsfonts,amssymb}
\usepackage{hyperref}
\usepackage{ifthen}

\makeatletter
\def\@seccntformat#1{\csname the#1\endcsname.\ } 
\def\@biblabel#1{#1.} 
\makeatother

\date{}

\newenvironment{proof}[1][\hspace{-1.0ex}]%
{\par\addvspace{1mm}{\sc Proof\hspace{1.0ex}{#1}.} }%
{\quad$\blacktriangle$\par\addvspace{1mm}}

\newif\ifNoRemark
\def\addtheorem#1#2#3#4{
\ifthenelse{\equal{#2}{}}{}%
{\ifthenelse{\expandafter\isundefined\csname the#2\endcsname}{\newcounter{#2}}{}}
\newenvironment{#1}[1][\global\NoRemarktrue]
{\par\addvspace{2mm plus 0.5mm minus 0.2mm}\noindent 
{\bf #3}\ifthenelse{\equal{#2}{}}{}%
{\refstepcounter{#2}{\bf ~\csname the#2\endcsname}}%
{\bf \vphantom{##1}\ifNoRemark.\ \else\ (##1).\fi}\begingroup #4}%
{\endgroup\par\addvspace{1mm plus 0.5mm minus 0.2mm}\global\NoRemarkfalse}
\expandafter\newcommand\csname b#1\endcsname{\begin{#1}}
\expandafter\newcommand\csname e#1\endcsname{\end{#1}}
}

\addtheorem{theorem}{thrm}{Theorem}{\sl}
\addtheorem{lemma}{lmm}{Lemma}{\sl}
\addtheorem{proposition}{prpstn}{Proposition}{\sl}
\addtheorem{corollary}{crll}{Corollary}{\sl}
\addtheorem{example}{exmpl}{Example}{}

\title{On the gaps of the spectrum of volumes of trades%
\footnote{%
This is the peer reviewed version of the following article: 
D. S. Krotov, On the gaps of the spectrum of volumes of trades,
Journal of Combinatorial Designs, 26(3):119--126, 2008, 
which has been published in final form at 
\url{https://doi.org/10.1002/jcd.21592},
01 November 2017. 
This article may be used for non-commercial purposes 
in accordance with Wiley Terms and Conditions 
for Use of Self-Archived Versions.%
}%
}
\author{Denis~S.~Krotov%
\thanks{Sobolev Institute of Mathematics, Novosibirsk 630090, Russia. E-mail: krotov@math.nsc.ru}%
}

\begin{document}

\maketitle

\begin{abstract}
A pair $\{T_0,T_1\}$ of disjoint collections of $k$-subsets (blocks) of a set 
$V$ of cardinality $v$ is called a $t$-$(v,k)$ trade or simply  a $t$-trade if 
every $t$-subset of $V$ is included in the same number of blocks of $T_0$ and 
$T_1$. The cardinality of $T_0$ is called the volume of the trade. Using the 
weight distribution of the Reed--Muller code, we prove the conjecture that for 
every $i$ from $2$ to $t$, there are no $t$-trades of volume greater than 
$2^{t+1}-2^i$ and less than $2^{t+1} - 2^{i-1}$ and derive restrictions on the 
$t$-trade volumes that are less than $2^{t+1} + 2^{t-1}$.
\end{abstract}
\section{Introduction}
Trades reflect possible differences between two combinatorial designs:
if $D'$ and $D''$ are $t$-$(v,k,\lambda)$ designs, then
the pair $\{ D'\backslash D'',D''\backslash D' \}$ 
is a $t$-trade. If there are no $t$-trades of volume $s$, then, of course, 
there are no two $t$-$(v,k,\lambda)$ designs differing in exactly $s$ blocks.
This gives one of the motivations to study 
the spectrum of volumes of $t$-trades,
besides that the problem of determining 
the possible sizes 
of combinatorial objects 
from some studied class 
is very natural itself.
The smallest volume $2^t$ of a $t$-$(v,k)$ trade was determined independently in 
\cite{Hwang:PhD,Hwang:86} and \cite{FranklPach:83}. 
In \cite{MahSol:1992},
it was proved that the second smallest volume is $2^t+2^{t-1}$ (the nonexistence of trades of some partial volumes between $2^t$ and $2^t+2^{t-1}$ was proved in \cite{Hwang:86} and \cite{Malik:88}).
In \cite{GrayRam:98}, the existence of $t$-trades of volume $2^{t+1}-2^i$
was shown 
for every $i$ from $0$ to $t$.
It was conjectured by Mahmoodian and Soltankhah \cite{Soltankhah:88,MahSol:1992}
and by Khosrovshahi and Malik \cite{Khos:1990:trades,KMT:1999:trades} 
that 
there are no other volumes of $t$-trades less than $2^{t+1}$.
This was proved in \cite{AsgSol:2009} for Steiner $t$-trades,
that is, with the additional restriction that every $t$-subset is included in at most one block of $T_0$ ($T_1$)
(the nonexistence of Steiner $t$-trades of volume greater than $2^t+2^{t-1}$ and less than $2^t+2^{t-2}$ was shown earlier in \cite{HoorKhos:2005}).
The nonexistence of trades for some partial values of the volume, supporting the conjecture, was considered in \cite{GrayRam:98:vkt}.
The opposite pole of the  volume spectrum of trades also
attracts attention of researchers: the well-known halving conjecture \cite{Hartman:halving} can be considered 
as a partial case of the problem of determining the maximum volume of a $t$-$(v,k)$ trade
(this maximum is conjectured to be $\frac12{v\choose k}$ whenever ${v-i\choose k-i}$ is even for all $i$ from $0$ to $t$).

The goal of the current paper is to prove the conjecture on small volumes for all simple (without repetitions of blocks) $t$-trades.
We do this
by showing that the Boolean characteristic function of every $t$-$(v,k)$ trade,
considered as a function over the $v$-cube,
belongs to the Reed--Muller code $\mathcal{RM}(v-t-1,v)$ 
and utilizing the known facts \cite{BerSlo69:RM} about the weight distribution of this code.
The more advanced study \cite{KTA:76} of the weight distribution of the Reed--Muller codes
allows to derive restrictions on the volume of a $t$-trade if this volume is less than $2.5\cdot 2^t$,
which is also reflected in the main theorem of this paper.

Similar approach was applied in \cite{Pot12:spectra} 
to derive restrictions on small cardinalities of
switching components of objects in the whole $v$-cube: 
perfect codes, equitable partitions, 
correlation immune functions, bent functions.
We also represent a $t$-trade as a pair of subsets of the $v$-cube;
moreover, we consider more general class of trades in the $v$-cube, 
which can be considered as trades of $\{1,\ldots,t\}$-designs in the Hamming scheme, 
in the sense of \cite[Sect.~3.4]{Delsarte:1973}.

Section~\ref{s:def} contains the definitions.
In Section~\ref{s:properties}, 
we consider some simple properties of the defined concepts and relations between them.
In Section~\ref{s:th}, we formulate and prove the main result of the paper.
In Section~\ref{s:struct}, we discuss the structure of small (of volume less than $2^{t+1}$) $t$-trades. 
We conclude that the hypothetical description of small $t$-trades can be considered as 
the characterization of a subclass of the class of so-called $t$-unitrades. 
While the small $t$-unitrades are characterized by the result of \cite{KasamiTokura:70},
the classification of that subclass will probably be more complicated.

\section{Definitions}\label{s:def}
Let $V=\{a_1,\ldots,a_v\}$ be a finite set 
of cardinality $v$; for simplicity we assume that $a_i=i$.
The subsets of $V$ will be associated with their characteristic $v$-tuples, 
e.g., $\{2,4,5\}=(0,1,0,1,1,0,0)=0101100$ for $v=7$.
The cardinality of a subset (the number of ones in the corresponding tuple) 
will be referred to as its \emph{size}.
The set of all subsets of $V$ is denoted by $2^V$,
while ${V}\choose k$ stands for the set of subsets of cardinality $k$.

The set $2^V$ is referred to as the \emph{$v$-cube}.
In the natural way, the \emph{$v$-cube} 
is considered as a $v$-dimensional vector space over the finite field $\mathbb F_2$ of order $2$.
Given $n \in \{0,\ldots,v\}$, an \emph{$n$-subcube} is the set
$\{(x_1,x_2,\ldots,x_v)\in 2^V\,:\, x_{l_i}=b_i, i=1,\ldots,v-n\}$
for some distinct coordinate numbers $l_1$, \ldots, $l_{v-n}$ from $\{1,\ldots,v\}$ 
and some constants $b_1$, \ldots, $b_{v-n}$ from $\{0,1\}$.
So, any $n$-subcube is an affine subspace obtained as a translation of the 
linear span of some $n$ elements of the standard basis 
$\{1\}$, \ldots, $\{v\}$.

By a \emph{$[t]$-trade}, where 
$[t] = \{0,1,\ldots,t\}$, 
we will mean a pair
$\{T_0,T_1\}$ of disjoint subsets of $2^V$ 
such that for every $i$ from $[t]$ every $i$-subset $\bar s$ of $V$ 
is included in the same number of subsets from $T_0$ and from $T_1$:
\begin{equation}\label{eq:trade}
  |\{\bar x \in T_0 \,:\, \bar s\subseteq \bar x\}|=
|\{\bar x \in T_1 \,:\, \bar s\subseteq \bar x\}|.  
\end{equation}
$T_0$ and $T_1$ are called the \emph{legs} of the trade,
their elements are referred to as \emph{blocks}, 
and the cardinality $|T_0|=|T_1|$ 
(the equality follows from (\ref{eq:trade}) with $\bar s = \emptyset$)
is known as the \emph{volume} of the trade.

A partial case of $[t]$-trades is the so-called \emph{$t$-$(v,k)$-trades}, 
where $t<k<v$, which are the $[t]$-trades $\{ T_0,T_1 \}$ with 
$T_0,T_1 \subset {{V}\choose k}$
(for this classical type of trades, 
the condition (\ref{eq:trade}) 
for all subsets $\bar s$ of size less than $t$ 
follows from that for all $t$-subsets).

\section{Properties of the $[t]$-trades}\label{s:properties}

The following four statements help to understand the definition of the $[t]$-trades 
and their connection with the classical $t$-$(v,k)$ trades.
That statements are not used in the proof of the main result,
but Corollary~\ref{c:t-to-tkv} demonstrates that any example of a $[t]$-trade can be 
``lifted'' to an example of a $t$-$(v,k)$ trade, which is a general important fact.

\begin{lemma}\label{l:subcubes}
The pair $\{ T_0,T_1 \}$ of two disjoint subsets of $2^V$ is a $[t]$-trade 
if and only if every $(v-t)$-subcube 
contains the same number of elements from $T_0$ and from $T_1$.
\end{lemma}
\begin{proof}
\emph{Only if}.
Let $s\le t$; consider a $(v-s)$-subcube
$S^{b_1,...,b_s}_{l_1,...,l_s}=\{(x_1,x_2,\ldots,x_v)\in 2^V\,:\, x_{l_i}=b_i, i=1,\ldots,s\}$.
We prove by induction on the number of zeros among  $b_1$, \ldots, $b_s$ that
\begin{equation}\label{eq:scubes}
 |S \cap T_0 | = |S \cap T_1|.
\end{equation}
If there are no zeros, then (\ref{eq:scubes}) is straightforward from the definition of a trade.
If there is a zero, we can assume that $b_s=0$.
Then, 
$$S^{b_1,...,b_s}_{l_1,...,l_s} = S^{b_1,...,b_{s-1},0}_{l_1,...,l_{s-1},l_s} 
= S^{b_1,...,b_{s-1}}_{l_1,...,l_{s-1}}\backslash S^{b_1,...,b_{s-1},1}_{l_1,...,l_{s-1},l_s}, $$
and (\ref{eq:scubes}) follows from the induction assumption.

\emph{If}. If (\ref{eq:scubes}) holds for all $(v-t)$-subcubes, 
then it readily holds for the $(v-s)$-subcubes for all $s\le t$. 
Then, $\{ T_0,T_1 \}$ is a trade by definition.
\end{proof}

\begin{corollary}\label{c:transl}
 Every translation of a $[t]$-trade obtained by adding the same vector to every block is also a $[t]$-trade. 
\end{corollary}

\begin{lemma}\label{l:dbl}
Assume that 
the pair $\{ T_0,T_1 \}$ 
of two subsets of $2^V$ 
is a $[t]$-trade and $i\in\{1,\ldots,v\}$.
Then the pair $\{  T'_0, T'_1 \}$ 
obtained from $\{ T_0,T_1 \}$ by replacing
every block $(x_1,\ldots,x_v)$ by $(x_1,\ldots,x_v,x_i)$ is a $[t]$-trade.
\end{lemma}
\begin{proof}
Without loss of generality we assume that $i=1$. 
A set  $\bar z=(z_1,\ldots,z_v,z_{v+1})$
is included in a block $(x_1,\ldots,x_v,x_1)$
if and only if 
$\bar z' = (z_1 \vee z_{v+1},z_2,\ldots,z_v)$,
where $z_1\vee z_{v+1} = z_{1}+z_{v+1}+z_{1}z_{v+1}$, 
is included in $(x_1,\ldots,x_v)$.
If $|\bar z|\le t$, then obviously $|\bar z'|\le t$;
so, $\{ T'_0, T'_1 \}$ being a $[t]$-trade follows from that of $\{ T_0,T_1 \}$, by the definition.
\end{proof}

\begin{corollary}\label{c:t-to-tkv}
Assume that 
the pair $\{ T_0,T_1 \}$ 
of two subsets of $2^V$ 
is a $[t]$-trade.
Then the pair $\{ \hat T_0,\hat T_1 \}$ 
formed from $\{ T_0,T_1 \}$ by replacing
every block $(x_1,\ldots,x_v)$ 
by $(x_1,\ldots,x_v,x_1+1,\ldots,x_v+1)$ is a $t$-$(2v,k)$ trade, where $k=v$.
\end{corollary}
\begin{proof}
Applying Lemma~\ref{l:dbl}, $v$ times for $i=1,2,\ldots,v$, we see that  
$\{ \tilde T_0,\tilde T_1 \}$ 
obtained from $\{ T_0,T_1 \}$ by doubling
$(x_1,\ldots,x_v,x_1,\ldots,x_v)$
every block $v$-tuple $(x_1,\ldots,x_v)$ is a $[t]$-trade.

Then,  $\{ \hat T_0,\hat T_1 \}$ is a $[t]$-trade by Corollary~\ref{c:transl}.
Since all blocks of  $\{ \hat T_0,\hat T_1 \}$ has the same size $v$ 
(the number of ones in the corresponding $2v$-tuple), 
it is a $t$-$(2v,v)$ trade.
\end{proof}

\section{The main result} \label{s:th}
\begin{theorem}\label{th:gaps}
If the volume of a $[t]$-trade is less than 
$2^{t+1}+2^{t-1}$, then it has one of the following forms:
\begin{enumerate}
\item[\rm (1)] 
 $2^{t+1}-2^{i}, \qquad i\in \{0,\ldots,t+1\}$,
\item[\rm (2)] 
$2^{t+1}+2^{i}, \qquad i\in \{\lceil\frac{t-1}2\rceil,\ldots,t-2\}$,
\item[\rm (3)] 
$2^{t+1}+2^{t-1}-2^i, \qquad i\in \{0,\ldots,t-1\}$,
\item[\rm (4)] 
$2^{t+1}+2^{t-1}-3\cdot 2^i, \qquad i\in \{0,\ldots,t-3\}$.
\end{enumerate}
\end{theorem}
So, in particular, the smallest $t+1$ non-zero volumes of $[t]$-trades
are 
$2^t$, 
$2^t(2-\frac 12)$, 
$2^t(2-\frac 1{2^2})$, \ldots,  
$2^t(2-\frac 1{2^t}) = 2^{t+1} - 1$ 
(trades of these volumes are known to exist, see \cite{GrayRam:98}).
The next volume is $2^{t+1}$ (case (3), $i=t-1$); 
a trade of such volume can be constructed as the union of two disjoint trades of volume $2^t$.
The existence of trades of other volumes from cases (2), (3), (4) remains unknown in general
(some partial values are known, e.g., there are $2$-trades of volume $9$ \cite{FGG:2004:trades}).
The smallest hypothetical volume of a $t$-trade is $2^{t+1}+2^{\lceil\frac{t-1}2\rceil}$ if $t\ge 5$ (case (2)),
and $2^{t+1}+2^{t-3}$ if $t\in\{3,4,5,6\}$ (case (4)).

\begin{proof}
A function over $2^V$ with values from $\mathbb F_2$
is called a Boolean \emph{function}.
The \emph{weight} of a Boolean function is the number
of its nonzeros (ones).
Two Boolean functions $f$ and $g$ are \emph{orthogonal} if 
$\sum_{\bar x \in 2^V}f(\bar x)g(\bar x) =0$, 
i.e., the number of their common ones is even.


The set of all Boolean functions $f(x_1,\ldots,x_v)$ 
representable as a polynomial of degree at most $r$ in $x_1$, \ldots, $x_v$ is known
as the \emph{Reed--Muller code} of order $r$, $\mathcal{RM}(r,v)$.
The codes
$\mathcal{RM}(r,v)$ and $\mathcal{RM}(v-r-1,v)$ are dual to each other:
$\mathcal{RM}(v-r-1,v)$ consists of all Boolean functions 
that are orthogonal to every function from $\mathcal{RM}(r,v)$, and vice versa,
see e.g. \cite{MWS}.

The following two lemmas summarizes the results from \cite{BerSlo69:RM} and \cite{KTA:76} concerning the possible weights of
the functions from $\mathcal{RM}(r,v)$ from the interval $(0,2.5\cdot 2^{v-r})$.
\begin{lemma}[\cite{BerSlo69:RM}] \label{l:BS} The $v-r$ smallest nonzero weights of the Boolean functions from
$\mathcal{RM}(r,v)$ are
\begin{enumerate}
\item[\rm (1)] 
$2\cdot 2^{v-r}-2^{v-r}$, \
$2\cdot 2^{v-r}-2^{v-r-1}$, \ 
\ldots, \
$2\cdot 2^{v-r}-2$ .
\end{enumerate}
\end{lemma}

\begin{lemma}[\cite{KTA:76}] \label{l:KTA}
The weights larger than 
$2\cdot 2^{v-r}$
and smaller than
$2.5\cdot 2^{v-r}$
of the Boolean functions from
$\mathcal{RM}(r,v)$ are:
\begin{enumerate}
\item[\rm (2)] $2\cdot 2^{v-r}+2\cdot 2^{v-r-l}$, \qquad $2\le l \le (v-r+2)/2$;
\item[\rm (3)] $2.5\cdot 2^{v-r} - 2\cdot 2^{i}$, \qquad $0\le i \le v-r-3$;
\item[\rm (4)] $2.5\cdot 2^{v-r} - 6\cdot 2^{i}$, \qquad $0\le i \le v-r-4$.
\end{enumerate}
\end{lemma}

Now, let us consider the Boolean characteristic function $\chi_{T_0\cup T_1}$ 
of the union $T_0\cup T_1$ of the legs of a $[t]$-trade $\{ T_0,T_1 \}$.
It is easy to see that for every binary $v$-tuple 
$(s_1,s_2,\ldots,s_v)$ representing an  $i$-subset $\bar s$, $i\le t$, 
the function $\chi_{T_0\cup T_1}$ is orthogonal to the monomial 
$M_{\bar s}=x_1^{s_1}x_2^{s_2}\ldots x_v^{s_v}$.
Indeed, as it follows from (\ref{eq:trade}), the number of ones of $M_{\bar s}$ in $T_0$ equals the number 
of ones of $M_{\bar s}$ in $T_1$; hence, the number of ones of $M_{\bar s}$ in $T_0\cup T_1$ is even 
and $M_{\bar s}$ and $\chi_{T_0\cup T_1}$ are orthogonal.
So, $\chi_{T_0\cup T_1}$ is orthogonal to every monomial of degree at most $t$ and, as a corollary, to 
every polynomial of degree at most $t$. 
This shows that $\chi_{T_0\cup T_1}$ belongs to $\mathcal{RM}(v-t-1,v)$. Utilizing 
Lemmas~\ref{l:BS}
and~\ref{l:KTA}
and noting that the volume of the trade is the half of the weight of $\chi_{T_0\cup T_1}$,
we get the statement of the theorem.
\end{proof}

\section{On the structure of trades of small volume}\label{s:struct}

As we see from the previous section, every $[t]$-trade corresponds to a codeword of the Reed--Muller code $\mathcal{RM}(v-t-1,v)$.
We are especially interested in the structure of $[t]$-trades of small volume, less than $2^{t+1}$; 
the codewords of $\mathcal{RM}(v-t-1,v)$ of the corresponding weights were completely characterized in \cite{KasamiTokura:70}.


\begin{lemma}[\cite{KasamiTokura:70}]\label{l:KT}
Any Boolean function $f$ from $\mathcal{RM}(r,v)$
of  weight greater than $2^{v-r} $ and less than $2\cdot 2^{v-r}$ 
can be reduced by an invertible affine transformation of its variables to one of the following forms:
\begin{enumerate}
 \item[\rm (A)]
$
f(y_1,\ldots,y_v) = y_1\dots y_{r-\mu}(y_{r-\mu+1}\dots y_r + y_{r+1}\dots y_{r+\mu}), 
$
\item[\rm (B)]
$
f(y_1,\ldots,y_v) = y_1\dots y_{r-2}(y_{r-1} y_m + y_{r+1}y_{r+2}+\ldots + y_{r+2\nu-3}y_{r+2\nu-2}),
$
\end{enumerate}
where $v\ge r+\mu$ and $r \ge \mu \ge 2$, $v\ge r-2+2\nu$ and $\nu \ge 3$.
Any Boolean function from $\mathcal{RM}(r,v)$
of minimum nonzero weight, $2^{v-r}$, is the characteristic function
of a $(v-r)$-dimensional affine subspace of $2^V$.
\end{lemma}

It is not difficult to understand that the functions from $\mathcal{RM}(v-t-1,v)$
are exactly the functions that have even number of ones in every $(v-t)$-subcube.
The set of ones of such a function is called a $t$-\emph{unitrade}
(in \cite{Pot12:spectra}, it was called a bitrade, which does not agree with the use of this term in the literature; 
in \cite{Potapov:2013:trade}, the term ``unitrade'' was introduced for similar objects corresponding to the latin trades).
The \emph{uni}on of the legs of a \emph{trade} is always a unitrade, but the reverse is not true.
We will say that a $t$-unitrade $T$ is \emph{splittable} if it can be split into a $[t]$-trade $\{T_0,T_1\}$, $T_0 \cup T_1 = T$.
So, Lemma~\ref{l:KT} completely characterizes all $t$-unitrades of cardinality less than $2^{t+2}$.
As we see, beyond the unitrades of minimum cardinality $2^{t+1}$, there are two types of them.

A unitrade of type (B) is an intersection of an affine subspace of dimension $t+3$ and the set of ones of a quadratic function.
The affine rank (the dimension of the affine span) of such unitrade is $t+3$. The existence of splittable unitrades of type (B) is an open problem.
Note that we excluded the case $\nu =2$ from type (B),
as it is covered by type (A), $\mu =2$, and examples of corresponding trades can be easily constructed.

A unitrade of type (A) is the symmetric difference of two intersecting affine subspaces of dimension $t+1$.
If the dimension of the intersection is $i$, $i<t$, then the cardinality of the unitrade is $2^{t+2}-2^{i+1}$
and its affine rank is $2t+2-i$.
The structure of two intersecting affine subspaces is rather clear in an abstract vector space. 
However, the possibility to split such a set into two parts $T_0$ and $T_1$ forming a trade essentially depends on the basis
(as we see from Lemma~\ref{l:subcubes} 
and the definition of a subcube).
So, even for the case 
of simple affine subspaces of dimension $t+1$, 
the situation is not trivial.

\begin{example}
 Both sets $T'=\{000, 011, 101, 110\}$ 
and $T''=\{000,\linebreak[2]011, \linebreak[2]100,\linebreak[2]111\}$ 
are two-dimensional linear subspaces of $2^V$, $|V|=3$, 
and so $\chi_{T'},\chi_{T''}\in \mathcal{RM}(v-t-1,v)$ for $v=3$, $t=1$. 
However, only the second one can be split into a trade, 
$\{T''_0,T''_1\}=\{\{000,111\},\linebreak[2]\{011,100\}\}$. 

Note that using Corollary~\ref{c:t-to-tkv}, every example of $[t]$-trade 
can be transformed to an example of $t$-$(2v,v)$ trade.

As we will see from the following proposition, 
the key property of $T''$ 
is that it is a linear span of disjoint blocks,
$011=\{2,3\}$ and $100=\{1\}$, 
while $T'$ has no such a basis.
\end{example}

\begin{proposition}\label{p:split}
An affine subspace $T \subset 2^V$ of dimension $t+1$ can be split into a $[t]$-trade $\{ T_0,T_1\}$
if and only if it is a translation of the linear span of mutually disjoint base subsets.
\end{proposition}
\begin{proof}
 \emph{If.} Assume without loss of generality that $T$ is a linear span of mutually disjoint (as subsets of $V$)
 tuples $\bar x_1$, \ldots, $\bar x_{t+1}$. So,
 $$ T = \{ \alpha_1\bar x_1+\ldots +\alpha_{t+1}\bar x_{t+1} \ :\ \alpha_1,\ldots,\alpha_{t+1}\in\{0,1\} \}.$$
 Let $T_0$ ($T_1$) consist of blocks from $T$ with even (odd, respectively) sum $\alpha_1+\ldots+\alpha_{t+1}$.
 Every $(v-t)$-subcube $S$ consists of the tuples with fixed values in some $t$ positions. 
 There is $i\in \{1,\ldots,t+1\}$ such that $\bar x_i$ has zeros in all these positions. 
 This means that if $\bar z \in S$, then $\bar z +\bar x_i  \in S$. 
 In particular, adding $\bar x_i$ is a one-to-one correspondence between $T_0\cap S$ and $T_1\cap S$.
 Hence, $|T_0\cap S|=|T_1\cap S|$.
 
 \emph{Only if.} Let $B$ be the set of minimal (by inclusion) nonempty tuples from $T$.
 If $B$ consists of mutually disjoint subsets, then it is easy to see that $B$ is a basis of $T$.
 Assume that at least two subsets from $B$ intersect. 
 Let $\bar y$ and $\bar z$ be two different intersecting subsets from $B$ with minimum $|\bar y\cup\bar z|$
 (among the all intersecting pairs of subsets from $B$).
 
 (*) \emph{We state that $\bar y+\bar z \in B$}.
 Indeed, if it is not so, 
 then $B$ contains a proper subset $\bar x$ of $\bar y+\bar z$.
 In this case, $\bar x$ intersects with $\bar y$ (otherwise $\bar x\subset \bar z$) 
 and with $\bar z$ (otherwise $\bar x\subset \bar y$), and one of the pairs 
 $(\bar y,\bar x)$, $(\bar z,\bar x)$ contradicts the minimality of $|\bar y\cup\bar z|$.
 
 (**) \emph{For every two elements from $\emptyset$, $\bar y$, $\bar z$, $\bar z+\bar y$, 
 there is a $(v-t)$-subcube intersecting with $T$ in these and only these two elements}. 
Because of the linearity, it is sufficient to prove the statement for
$\emptyset$ and any $\bar x$ from  $\bar y$, $\bar z$, $\bar z+\bar y$.
Let $(\bar e_1, \ldots, \bar e_v)$ be the standard basis in $2^V$ 
(that is, $\bar e_i= \{i\}$, $i = 1,\ldots,v$),
and 
let $\bar x = \{i_0,\ldots,i_l\} = \bar e_{i_0}+ \ldots + \bar e_{i_l}$, where $l+1$ is the size of $\bar x$.
As follows from (*), the $l$-subcube $S_l$ spanned by $\bar e_{i_1}, \ldots , \bar e_{i_l}$ intersects with $T$ trivially, 
i.e., in only one element $\emptyset = (0,\ldots,0)$. Then, the dimension of the linear span of $S_l$ and $T$ is
$l+t+1$. 
If this value is less than $v$, than there is $i_{l+1}$ such that $\bar e_{i_{l+1}}$ 
does not belong to the linear span of $S_l$ and $T$. We add $\bar e_{i_{l+1}}$ to the basis $\bar e_{i_1}, \ldots , \bar e_{i_l}$
obtaining a subcube $S_{l+1}$. Again, if $l+t+2 < v$, then we can add some $\bar e_{i_{l+2}}$, and so on.
On the $(v-t-l-1)$th step, we will obtain a $(v-t-1)$-subcube that intersects with $T$ only in $\emptyset$.
Then, adding $\bar e_{i_{0}}$ to the basis leads to a $(v-t)$-subcube that intersects with $T$ only in $\emptyset$ and $\bar x$.

The rest of the proof is rather obvious.
If we have a partition $\{ T_0,T_1 \}$ of $T$, 
then at least two of
$\bar y$, $\bar z$, $\bar y + \bar z$ belong to the same cell of the partition, $T_0$ or $T_1$.
As follows from (**) and Lemma~\ref{l:subcubes}, such partition $\{ T_0,T_1 \}$ cannot be a trade.
\end{proof}

So, every $[t]$-trade of minimum volume is a partition of 
some affine subspace $T$, which is the translation by some block $\bar w$ of a linear subspace.
It is rather clear that if such trade consists of the blocks of the same size $k$,
then $\bar w$ intersects with every base subset $x_j$ in exactly $|x_j|/2$ elements, 
which means that $|x_j|$ must be even. 
This gives an alternative point of view to the characterization \cite{KMW:92} of $t$-$(v,k)$
trades of volume $2^t$.

Using the following well-known and easy fact, it is possible to construct a $[t]$-bitrade of volume $2^{t+1}-2^{i}$ 
from $[t]$-bitrades  of volume $2^{t}$ for every $t\ge 1$ and $i<t$, see Proposition~\ref{ex:321} below
(essentially, in view of Corollary~\ref{c:t-to-tkv}, it is a result of \cite{GrayRam:98}).
\begin{lemma}\label{l:sum}
 Assume that $\{T_0,T_1\}$ and  $\{T'_0,T'_1\}$ are two different $[t]$-trades 
 such that $T_0 \cap T'_0 =T_1 \cap T'_1 = \emptyset$.
 Then $\{(T_0 \cup T'_0)\backslash (T_1 \cup T'_1), (T_1 \cup T'_1)\backslash (T_0 \cup T'_0) \}$ 
 is a $[t]$-trade.
\end{lemma}
\begin{proposition}\label{ex:321}
 For every $t>0$ and $i<t$, $2t+2-i\le v$, the symmetrical difference of two $(t+1)$-subcubes 
\begin{eqnarray*}
 && T=\langle \{1\},\{2\},\ldots,\{i\},\{i+1\},\{i+2\},\ldots,\{t+1\} \rangle,\\
 && T'=\langle \{1\},\{2\},\ldots,\{i\},\{t+2\},\{t+3\},\ldots,\{2t+2-i\} \rangle
\end{eqnarray*}
 is a splittable type-(A) $t$-unitrade of cardinality $2^{t+2}-2^{i+1}$
 (the volume of the corresponding trade is $2^{t+1}-2^{i}$).
\end{proposition}
\begin{proof}
Let $T_0$ ($T_1$) consist of sets of even (odd) cardinality of $T$;
let $T'_0$ ($T'_1$) consist of sets of odd (even) cardinality of $T'$.
\end{proof}

So, some splittable unitrades of type (A) can be constructed as the symmetrical difference of minimum splittable unitrades. 
However, there are splittable unitrades of type (A) that cannot be treated in such a way:

\begin{example}\label{ex:had}
Consider two linear subspaces 
\begin{eqnarray*}
 C_0&{=}&\{0000000,\linebreak[2] 0011101,\linebreak[2] 0111010,\linebreak[2] 1110100,\linebreak[2] 1101001,\linebreak[2] 1010011,\linebreak[2] 0100111,\linebreak[2] 1001110\},\\
 C_1&{=}&\{0000000,\linebreak[2] 0010111,\linebreak[2] 0101110,\linebreak[2] 1011100,\linebreak[2] 0111001,\linebreak[2] 1110010,\linebreak[2] 1100101,\linebreak[2] 1001011\}
 \end{eqnarray*}
 (binary simplex codes of length $7$).
It is not difficult to check that each of them is a non-splittable $2$-unitrade.
 Their symmetrical difference is a $2$-unitrade, which can be split to the $2$-$(7,4)$ trade 
 $\{C_0 \backslash \{0000000\}, C_1 \backslash \{0000000\} \} $.
\end{example}

In the end of this section, we conclude that while the small $t$-unitrades 
are characterized by the result of \cite{KasamiTokura:70},
the characterization of small $[t]$-trades and $t$-$(v,k)$ trades, in particular, 
the existence of trades of type (B), remains to be an open problem.
Similar things happen with some latin trades \cite{Potapov:2013:trade}, 
which supports the general principle that unitrades are more easy to characterize than trades.


\section*{Acknowledgements}
The research was carried out at the Sobolev Institute of Mathematics
at the expense of the Russian Science Foundation (project No. 14-11-00555).

The author is grateful to Professor Gholamreza Khosrovshahi for introducing the conjecture solved 
in this paper and for the hospitality in the Institute for Research in Fundamental Science (IPM) in Tehran.



\begin{thebibliography}{10}

\bibitem{AsgSol:2009}
M.~Asgari and N.~Soltankhah.
\newblock On the non-existence of some {S}teiner $t$-$(v,k)$ trades of certain
  volumes.
\newblock {\em Util. Math.}, 79:277--283, July 2009.

\bibitem{BerSlo69:RM}
E.~R. Berlekamp and N.~J.~A. Sloane.
\newblock Restrictions on weight distribution of {R}eed--{M}uller codes.
\newblock {\em
  \href{http://www.sciencedirect.com/science/journal/00199958}{Inf. Control}},
  14(5):442--456, 1969.
\newblock \DOI{10.1016/S0019-9958(69)90150-8}.

\bibitem{Delsarte:1973}
P.~Delsarte.
\newblock {\em An Algebraic Approach to Association Schemes of Coding Theory},
  volume~10 of {\em Philips Res. Rep., Supplement}.
\newblock 1973.

\bibitem{FGG:2004:trades}
A.~D. Forbes, M.~J. Grannell, and T.~S. Griggs.
\newblock Configurations and trades in {S}teiner triple systems.
\newblock {\em \href{http://ajc.maths.uq.edu.au}{Australas. J. Comb.}},
  29:75--84, 2004.

\bibitem{FranklPach:83}
P.~Frankl and J.~Pach.
\newblock On the number of sets in a null $t$-design.
\newblock {\em
  \href{http://www.sciencedirect.com/science/journal/01956698}{Eur. J. Comb.}},
  4(1):21--23, 1983.
\newblock \DOI{10.1016/S0195-6698(83)80004-3}.

\bibitem{GrayRam:98}
B.~D. Gray and C.~Ramsay.
\newblock On a conjecture of {M}ahmoodian and {S}oltankhah regarding the
  existence of $(v,k,t)$ trades.
\newblock {\em Ars Combin.}, 48:191--194, April 1998.

\bibitem{GrayRam:98:vkt}
B.~D. Gray and C.~Ramsay.
\newblock On the spectrum of $[v,k,t]$ trades.
\newblock {\em \href{http://www.sciencedirect.com/science/journal/03783758}{J.
  Stat. Plann. Inference}}, 69(1):1--19, June 1998.
\newblock \DOI{10.1016/S0378-3758(97)00123-7}.

\bibitem{Hartman:halving}
A.~Hartman.
\newblock Halving the complete design.
\newblock {\em
  \href{http://www.sciencedirect.com/science/bookseries/01675060}{Ann. Discrete
  Math.}}, 34:207--224, 1987.

\bibitem{HoorKhos:2005}
A.~Hoorfar and G.~B. Khosrovshahi.
\newblock On the nonexistence of {S}teiner $t$-$(v,k)$ trades.
\newblock {\em Ars Combin.}, 75:195--204, 2005.

\bibitem{Hwang:PhD}
H.~L. Hwang.
\newblock {\em Trades and the Construction of {BIB} Designs with Repeated
  Blocks}.
\newblock {PhD} dissertation, University of Illinois, Chicago, 1982.

\bibitem{Hwang:86}
H.~L. Hwang.
\newblock On the structure of $(v,k,t)$ trades.
\newblock {\em \href{http://www.sciencedirect.com/science/journal/03783758}{J.
  Stat. Plann. Inference}}, 13:179--191, 1986.
\newblock \DOI{10.1016/0378-3758(86)90131-X}.

\bibitem{KasamiTokura:70}
T.~Kasami and N.~Tokura.
\newblock On the weight structure of {R}eed--{M}uller codes.
\newblock {\em
  \href{http://ieeexplore.ieee.org/xpl/RecentIssue.jsp?punumber=18}{IEEE Trans.
  Inf. Theory}}, 16(6):752--759, 1970.
\newblock \DOI{10.1109/TIT.1970.1054545}.

\bibitem{KTA:76}
T.~Kasami, N.~Tokura, and S.~Azumi.
\newblock On the weight enumeration of weights less than $2.5d$ of
  {R}eed--{M}uller codes.
\newblock {\em
  \href{http://www.sciencedirect.com/science/journal/00199958}{Inf. Control}},
  30(4):380--395, 1976.
\newblock \DOI{10.1016/S0019-9958(76)90355-7}.

\bibitem{Khos:1990:trades}
G.~B. Khosrovshahi.
\newblock On trades and designs.
\newblock {\em
  \href{http://www.sciencedirect.com/science/journal/01679473}{Comput. Stat.
  Data Anal.}}, 10(2):163--167, 1990.
\newblock \DOI{10.1016/0167-9473(90)90061-L}.

\bibitem{KMT:1999:trades}
G.~B. Khosrovshahi, H.~R. Maimani, and R.~Torabi.
\newblock On trades: an update.
\newblock {\em
  \href{http://www.sciencedirect.com/science/journal/0166218X}{Discrete Appl.
  Math.}}, 95(1--3):361--376, 1999.
\newblock \DOI{10.1016/S0166-218X(99)00086-4}.

\bibitem{KMW:92}
G.~B. Khosrovshahi, D.~Majumdar, and M.~Widel.
\newblock On the structure of basic trades.
\newblock {\em J. Comb. Inf. Syst. Sci.}, 17(1--2):102--107, 1992.

\bibitem{MWS}
F.~J. MacWilliams and N.~J.~A. Sloane.
\newblock {\em The Theory of Error-Correcting Codes}.
\newblock Amsterdam, Netherlands: North Holland, 1977.

\bibitem{MahSol:1992}
E.~S. Mahmoodian and N.~Soltankhah.
\newblock On the existence of $(v,k,t)$ trades.
\newblock {\em \href{http://ajc.maths.uq.edu.au}{Australas. J. Comb.}},
  6:279--291, 1992.

\bibitem{Malik:88}
F.~Malik.
\newblock {\em On $(v,k,t)$ trades}.
\newblock M.s. thesis, University of Tehran, Tehran, 1988.

\bibitem{Pot12:spectra}
V.~N. Potapov.
\newblock Cardinality spectra of components of correlation immune functions,
  bent functions, perfect colorings, and codes.
\newblock {\em \href{http://link.springer.com/journal/11122}{Probl. Inf.
  Transm.}}, 48(1):46--54, 2012.
\newblock \DOI{10.1134/S003294601201005X}, translated from
  \href{http://www.mathnet.ru/php/journal.phtml?jrnid=ppi\&option_lang=eng}{Probl.
  Peredachi Inf.} 48(1) 54--63, 2012.

\bibitem{Potapov:2013:trade}
V.~N. Potapov.
\newblock Multidimensional {L}atin bitrades.
\newblock {\em \href{http://link.springer.com/journal/11202}{Sib. Math. J.}},
  54(2):317--324, 2013.
\newblock \DOI{10.1134/S0037446613020146}, translated from
  \href{http://www.mathnet.ru/php/journal.phtml?jrnid=smj\&option_lang=eng}{Sib.
  Mat. Zh.} 54(2):407--416, 2013.

\bibitem{Soltankhah:88}
N.~Soltankhah.
\newblock {\em Investigation of existence and non-existence of some
  $(v,k,t)$-trades}.
\newblock M.s. thesis, Sharif University of Technology, Tehran, 1988.

\end{thebibliography}

\providecommand\href[2]{#2} \providecommand\url[1]{\href{#1}{#1}}
  \def\DOI#1{{\small {DOI}:
  \href{http://dx.doi.org/#1}{#1}}}\def\DOIURL#1#2{{\small{DOI}:
  \href{http://dx.doi.org/#2}{#1}}}

\end{document}